\documentclass[reqno]{amsart}
\usepackage{amsfonts,amsthm,amsmath,amssymb,fullpage,graphicx,bm,xcolor,bbm,mathrsfs,algorithm2e}
\usepackage[utf8]{inputenc}

\usepackage[british]{babel}
\usepackage[shortlabels]{enumitem}
\usepackage[font=sf]{caption}

\allowdisplaybreaks

\usepackage{hyperref} 
\hypersetup{colorlinks=true, linkcolor=blue, citecolor=blue, filecolor=blue, urlcolor=blue}
\usepackage[numbers, comma, sort&compress]{natbib}

\usepackage[shortlabels]{enumitem}
\usepackage[comma,sort&compress]{natbib}

\definecolor{greennew}{rgb}{0.1
,0.4,0.1}
 \title{\bf{Component Spectrum of Large Sparse Uniformly Random Magical Squares}}

  \author{Souvik Ray}\thanks{School of Data Science \& Society, University of North Carolina at Chapel Hill. souvikr@unc.edu}

\numberwithin{equation}{section}

\date{}
\usepackage{cleveref}
\newtheorem{theorem}{Theorem}[section]
\newtheorem{definition}[theorem]{Definition}
\newtheorem{result}[theorem]{Result}
\newtheorem{proposition}[theorem]{Proposition}
\newtheorem{corollary}[theorem]{Corollary}
\newtheorem{remark}[theorem]{Remark}
\newtheorem{lemma}[theorem]{Lemma}

\crefname{equation}{}{}
\crefname{theorem}{Theorem}{Theorems}
\crefname{assumption}{Assumption}{Assumptions}
\crefname{remark}{Remark}{Remarks}
\Crefname{lemma}{Lemma}{Lemmas}
\crefname{lemma}{Lemma}{Lemmas}
\crefname{enumi}{}{}

\begin{document}

\begin{abstract}
In this paper, we shall try to deduce asymptotic behaviour of component spectrum of random $n \times n$ magical squares with line sum $r \in \mathbb{N}$, which can also be identified as $r$-regular bipartite graphs on $2n$ vertices, chosen uniformly from the set of all possible such squares as the dimension $n$ grows large keeping $r$ fixed. We shall focus on limits (after appropriate centering and scaling) of various statistic depending upon the component structure,  e.g., number of small components, size of the smallest and largest components, total number of components etc. We shall observe that for the case $r=2$, this analysis falls into the domain of Logarithmic combinatorial structures as discussed in \cite{arratiabook}, although we shall present a new approach for this case relying only on the asymptotic results for random permutations which also helps us to demonstrate an importance sampling algorithm to estimate parameters defined in terms of uniform distribution on magical squares. The case $r \geq 3$ although does not fall in the domain of the Logarithmic combinatorial structures,   we shall establish that its component structure is rather trivial, using techniques based on a power series approach.
\end{abstract}



\maketitle

\section{Introduction}
The main focus of our study in this project is a special type of integer matrix, popularly known as \textit{magical squares}. A magical square of dimension $n \in \mathbb{N}$ and line sum $r \in \mathbb{N} \cup \left\{0\right\}$ is an $n \times n$ matrix with non-negative integer entries such that each of its row and column sums is equal to $r$. We shall denote by $\mathcal{M}(n,r)$ the collection of all such matrices, i.e.
\begin{equation}{\label{def}}
\mathcal{M}(n,r) := \left\{A = (a_{ij})_{i,j \in [n]} \bigg \rvert a_{ij} \in \mathbb{N} \cup \left\{0\right\}, \; \sum_{l \in [n]}a_{il} = \sum_{k \in [n]} a_{kj} =r, \; \forall \; i,j \in [n] \right\}, 
\end{equation}
where $[n]$ will denote the set $\left\{1, \ldots,n\right\}$ henceforth. For example $\mathcal{M}(n,0)$ consists of only the the $n \times n$ zero matrix whereas $\mathcal{M}(n,1)$ consists of $n!$ many $n \times n$ permutation matrices.  These kind of matrices were first studied by \cite{mcmohan}. Since then a vast literature has emerged trying to count the total number of such magical squares with dimension $n$ and line sum $r$, i.e. cardinality of the set $\mathcal{M}(n,r)$, which we denote by $H(n,r)$. A careful review of the history of the development in computing $H(n,r)$ can be found in \cite{stanley1973}, \cite[Chapter 1]{stan2} and \cite[Chapter 4]{stanenu}. Although the exact computation of $H(n,r)$ still remains out of reach except for some small values of $n$ and $r$, several asymptotic approximations of $H(n,r)$ have been proved over the course of time, albeit assuming some growth conditions on $n$ and $r$.  Three published papers in early seventies \cite{everett, bekessy1,bender} provided the asymptotic behaviour of $H(n,r)$ as $n \to \infty$ keeping $r$ bounded. Some recent advances in this directions appeared in \cite{canfield1} and \cite{canfield2}, where the authors derived asymptotic formulas for $H(n,r)$ for dense (i.e. the case when $r/n$ is bounded away from $0$) and sparse (for the case where $r=o(\sqrt{n})$) magical squares respectively. 

Elements of the set $\mathcal{M}(n,r)$ can also be identified as labelled balanced regular bipartite graphs of degree $r$ in the following way : for $A =(a_{ij})_{i,j \in [n]} \in \mathcal{M}(n,r)$, we identify this matrix with the bipartite graph $(V,E)$ with $2n$ vertices and having vertex bipartition $V = R \cup C$, where $R=\left\{R_1, \ldots, R_n\right\}, C=\left\{C_1,\ldots,C_n\right\}$ and the edge set $E$ contains $a_{ij} \in [0,r]$ many edges connecting vertices $R_i$ and $C_j$ for any $i,j \in [n]$. For example, the elements of $\mathcal{M}(n,1)$ corresponds to collections of all perfect bipartite matching with $2n$ vertices. This correspondence and the consequences of it will be crucial in our analysis helping to visualise certain constructions easily.

Bipartite graphs are one of the most popular object of study in graph theory, computer science and coding theory and have found applications in biology and medicine, for example in representation of enzyme-reaction links in metabolic pathways. On the other hand, magical squares arise frequently in different fields of mathematics and statistics. \textit{Contingency tables} or \textit{frequency tables} are one of the most used data structure in statistics for representing the joint empirical distribution of multivariate data and are useful for testing properties such as independence between row and columns. Inference using contingency tables requires analysis of the conditional distribution of a matrix with specified row and column sums; magical squares being one particular kind of such matrices. Finding an efficient algorithm to generate an uniformly distributed element from the set of all contingency tables with specified row and column sums and the deeply connected problem of calculating the total number of elements in the afore-mentioned set has attracted a lot of focus in statistical literature. We refer the reader to the work of \cite{gail,persi,greselin} for an account of the classical approaches to these problems. With all these applications in mind, it make sense to ask the question \textit{how does a typical large sparse magical square/regular bipartite graph look like? } Here sparse refer to the fact that we shall keep $r$ fixed and grow $n$ to infinity and thus most of the entries of the matrix will be zeroes.  Answering this question for a restricted case from a probabilistic point of view will be the main focus of this project.

To address the above mentioned question, we shall concentrate on the component structure of the magical squares. For any $A \in \mathcal{M}(n,r)$ and $k \in [n]$, we call a pair $(S,T)$ of non-empty subsets of $[n]$ a \textit{$k$-component} of $A$ if the following two conditions are satisfied. 

\begin{enumerate}
\item $\operatorname{card}(S)=\operatorname{card}(T)=k$.
\item Let $A(S,T)$ be the $k \times k$ submatrix of $A$ whose rows are indexed by $S$ and the columns are indexed by $T$, i.e. $A(S,T)=(a_{ij})_{i \in S, j \in T}$. Then every row and column of $A(S,T)$ sums to $r$, i.e., $A(S,T) \in \mathcal{k,r}.$ 
\end{enumerate}

We call $(S,T)$ a \textit{component} of $A$ if it is a \textit{$k$-component} for some $k$. A component $(S,T)$ is called \textit{irreducible} if any component $(S^{\prime},T^{\prime})$ with $S^{\prime} \subseteq S$ and $T^{\prime} \subseteq T$ satisfies $S^{\prime}=S$ and $T^{\prime}=T$. The matrix $A(S,T)$ is then also called \textit{irreducible}.  For example, $(\left\{i\right\}, \left\{j\right\})$ is a $1$-component (in which case it is also irreducible) if and only if $a_{ij}=r$. It is easy to observe that if $G_A$ is the regular bipartite graph corresponding to the matrix $A$, then the irreducible components of $A$ corresponds to connected components of the graph $G_A$, i.e., if $(S,T)$ is an irreducible component of $A$ then the vertex set $\left\{R_i : i \in S\right\} \cup \left\{C_j : j \in T\right\}$ forms a connected component of $G_A$. Thus an irreducible $A \in \mathcal{M}(n,r)$ corresponds to a connected regular bipartite graph $G_A$ of degree $r$ with $2n$ vertices. Henceforth, $\chi_k(A)$ will denote the number of $k$-components of magical square $A$ or equivalently the number of connected components of size $2k$ for the bipartite graph $G_A$. 

Fix $r \in \mathbb{N} \cup \left\{0\right\}$.  Consider an uniform magical square $\mathcal{A}_n \sim \text{Uniform}(\mathcal{M}(n,r))$. We want to address the following two questions in this project.
\begin{enumerate}[label=(\alph*)]
\item How does the joint distribution of $(\chi_1(\mathcal{A}_n), \chi_2(\mathcal{A}_n), \ldots)$ behave as $n \to \infty$?
\item What is the asymptotic distribution (if any after proper centering and normalising) of the size of the smallest and largest component of $\mathcal{A}_n$, i.e., $S_n := \inf \left\{ k \geq 1 : \chi_k(\mathcal{A}_n) >0 \right\}$ and $L_n := \sup \left\{ k \geq 1 : \chi_k(\mathcal{A}_n) >0 \right\}$ respectively?
\item What is the asymptotic behaviour of total number of irreducible components of $\mathcal{A}_n$, i.e., $C_n := \sum_{i=1}^n \chi_i(\mathcal{A}_n)$ ?
\end{enumerate} 
For the case $r=1$, there are $n$ many $1$-component for any member of $\mathcal{M}(n,1)$. This being a trivial case, we shall only focus on the situation when $r \geq 2$. It will turn out that the case $r \geq 3$ is also trivial in the sense that most of the matrices in $\mathcal{M}(n,r)$ are irreducible, see \Cref{rgeq3}. Only the case $r=2$ will give us some non-trivial results.

\section{Component Spectrum Asymptotics Using Formal Power Series}
\label{power}

In this section, we shall derive asymptotic properties of the small components counts, the size of the smallest component and total number of components for uniformly random magical squares with line sum$r$ and dimension $n$. Our approach will be to exploit a formal power series identity, see \Cref{cycleindexm}, relating the the number of magical squares in $\mathcal{M}(n,r)$ and having a pre-specified component spectrum with number of irreducible magical squares in $\mathcal{M}(n,r)$. We emphasize on the fact that the power series' are formal in many situations since they have radii of convergence $0$ which will be the situation for $r \geq 3$ especially; the identities are only valid in the sense of equality of co-efficients. Nevertheless, we can still differentiate those formal power series' by defining the derivative of the formal power series $\sum_{n \geq 0} p_nx^n$ to be $\sum_{n \geq 1} np_nx^{n-1}$.   These operations will be routinely carried out in this section. We introduce the notation $\textsc{Coeff}(\sum_{n \geq 0} p_n x^n ; m)$ to denote the $m$-th co-efficient of the power series, i.e., $p_m$.

We start this section by recalling the classical \textit{Cycle index theorem} for symmetric groups. This will motivate our approach presented in this section and later will be instrumental in our analysis in \Cref{symm}. Let $\mathscr{S}_n$ be the symmetric group of order $n$ containing all the permutations of the set $[n]$ and for any $\pi \in \mathscr{S}_n$, we shall denote by $c_i(\pi)$ the number of cycles of length $i \in [n]$ present in the permutation $\pi$. The cycle index theorem states that for any $x, t_1, t_2, \ldots,$ we have the equality of the following two formal power series.
\begin{equation}{\label{cycleindex}}
\sum_{n \geq 0} \sum_{\pi \in \mathscr{S}_n} t_1^{c_1(\pi)}\ldots t_n^{c_n(\pi)} \dfrac{x^n}{n!} = \exp \left( \sum_{n \geq 1} \dfrac{t_nx^n}{n}\right).
\end{equation}
For any $a_1, \ldots,a_n \in \mathbb{N} \cup \left\{0\right\}$, if we define $h^{(n)}(a_1,\ldots,a_n)$ as the number of permutations $\pi \in \mathscr{S}_n$ with exactly $a_i$ many $i$-cycles; then \Cref{cycleindex} can be rewritten as 
  \begin{equation}{\label{cycleindex2}}
  \sum_{n \geq 0} \sum_{a_1, \ldots, a_n \geq 0} h^{(n)}(a_1,\ldots,a_n) t_1^{a_1}\ldots t_n^{a_n} \dfrac{x^n}{n!} = \exp \left( \sum_{n \geq 1} \dfrac{t_nx^n}{n}\right).
\end{equation}
Based on this formal power series representation, one can prove the following theorem for asymptotic cycle structure of a random permutation.

\begin{theorem}{\label{cycleindex3}}
\begin{enumerate}[label=(\Alph*)]
\item Consider a geometric random variable $N_x$ with failure probability $x \in (0,1)$, i.e., $\mathbb{P}(N_x=n) = (1-x)x^n$, for all $n \geq 0$. Let $\Pi_x \rvert (N_x=n) \sim \text{Uniform}(\mathscr{S}_n).$ Then 
$$ \left(c_1(\Pi_x), c_2(\Pi_x), \ldots \right) \stackrel{d}{=} (Z_{1,x}, Z_{2,x}, \ldots),$$
where $Z_{i,x} \stackrel{ind.}{\sim} \text{Poisson}(x^i/i)$, for all $i \geq 1$.

\item Let $\Pi_n \sim \text{Uniform}(\mathscr{S}_n)$. Then 
$$ \left(c_1(\Pi_n), c_2(\Pi_n), \ldots \right) \stackrel{d}{\longrightarrow} (Z_1, Z_2, \ldots),$$
where $Z_i \stackrel{ind.}{\sim} \text{Poisson}(1/i)$, for all $i \geq 1$. 

\end{enumerate}
\end{theorem}

The second assertion in \Cref{cycleindex3} is proved using its first assertion and application of a proper Tauberian theorem. For a short proof of \Cref{cycleindex3}, a general history of this result and an estimate of total variation distance between the two processes appearing in the statement of \Cref{cycleindex3}, we refer to the work of \cite{arratia}. 

The second assertion of \Cref{cycleindex3} is a special case of a general class known as \textit{Logarithmic Combinatorial Structures}. At the heart of this concept lies a combinatorial structure that can be decomposed into component elements, like cycle decomposition for permutations or connected components for a graph. In such a decomposable structure of size $n$, we have $C_i^{(n)}$ many components of size $i \in [n]$, satisfying $\sum_{i=1}^n iC_i^{(n)}=n$. 
\begin{definition}{\label{log}}
We say that a decomposable structure satisfies the \textit{Logarithmic Condition} if there exists a sequence of independent non-negative integer valued random variables $\left\{Z_i : i \geq 1\right\}$ such that 
\begin{enumerate}
\item $\left( C_1^{(n)}, \ldots, C_n^{(n)} \right) \stackrel{d}{=} \left( Z_1, \ldots, Z_n \right) \bigg \rvert \left( \sum_{i=1}^n iZ_i=n\right),$
\item there exists $\theta \in (0, \infty)$ such that $i\mathbb{P}(Z_i=1), i \mathbb{E}Z_i \to \theta$ as $i \to \infty$. 
\end{enumerate}
\end{definition}

We refer the interested reader to the  excellent monograph~\cite{arratiabook} on Logarithmic Combinatorial Structures and the asymptotic behaviour of their component spectrum in order to get a glimpse of some of the universal properties shared by them. In particular, the number of small components, i.e.  $(C_1^{n}, \ldots, C_k^{(n)})$ for any fixed $k$, are asymptotically independent with $C_i^{(n)}$ converging weakly to the distribution of $Z_i$. A lot of well-known structures fall into the category described by the Logarithmic condition, permutations being one of them. In fact the formula \Cref{cycleindex} is derived from the following identity,
$$ \#( \pi \in \mathscr{S}_n \text{ with } c_1(\pi)=c_i \text{ for all } i=1, \ldots,n) = \dfrac{n!}{\prod_{i=1}^n (i!)^{c_i}c_i!} \prod_{i=1}^n ((i-1)!)^{c_i}  = \dfrac{n!}{\prod_{i=1}^n i^{c_i}c_i!}, \; \text{if } \sum_{i=1}^n ic_i=n.$$
This shows that permutations satisfy \Cref{log} with $Z_i \sim \text{Poisson}(1/i)$ and $\theta=1$. 

We start the analysis of $\mathcal{M}(n,r)$ by first computing the number of such matrices/graphs with pre-specified component spectrum. Denote by $f(n,r)$ the total number of irreducible magical square/ connected regular bipartite graphs in $\mathcal{M}(n,r)$. Fix $a_1, \ldots, a_n \in \mathbb{N} \cup \left\{0\right\}$ with $\sum_{i=1}^n ia_i=n$. The number of ways in which $n$ rows can be partitioned such that there are $a_i$ many subsets with $i$ many elements is $n!/(\prod_{i=1}^n (i!)^{a_i}a_i!)$. The corresponding columns then can be chosen in $n!/(\prod_{i=1}^n (i!)^{a_i})$ many ways. Once we have partitioned the rows and columns, there are $f(i,r)$ many ways to assign an irreducible matrix to a particular component of $i$ rows and $i$ columns. Hence, if  $h^{(n)}_r(a_1, \ldots,a_n)$ denotes the number of $A \in \mathcal{M}(n,r)$ having exactly $a_i$ many $i$-components, we have
\begin{equation}{\label{cycindexm}}
h^{(n)}_r(a_1, \ldots,a_n) = \dfrac{n!^2}{\prod_{i=1}^n (i!)^{2a_i}a_i!} \prod_{i=1}^n (f(i,r))^{a_i}, \; \text{provided } \sum_{i=1}^n ia_i =n.
\end{equation} 
Summing over $a_1, \ldots,a_n$ in the equation on \Cref{cycindexm}, we have the following formal power series identity, see \cite[Proposition 5.5.8]{stanenu} for discussion and proof regarding this identity.
\begin{equation}{\label{cycleindexm}}
      \sum_{n \geq 0} \sum_{a_1, \ldots, a_n \geq 0} h^{(n)}_r(a_1,\ldots,a_n) t_1^{a_1}\ldots t_n^{a_n} \dfrac{x^n}{(n!)^2} = \exp \left( \sum_{n \geq 1} \dfrac{f(n,r)t_nx^n}{(n!)^2}\right).
 \end{equation}
Plugging in $t_1=t_2=\cdots=1$ in \Cref{cycleindexm}, we get the following auxiliary identity. 
\begin{equation}{\label{cycleindexm2}}
      \sum_{n \geq 0} H(n,r) \dfrac{x^n}{(n!)^2} = \exp \left( \sum_{n \geq 1} \dfrac{f(n,r)x^n}{(n!)^2}\right).
 \end{equation}
Fortunately, it is not hard to find exact values of $f(n,2)$ for any positive integer $n$. Indeed a small argument, as presented in the proof of \cite[Proposition 5.5.10]{stanenu}, shows that $f(n,2)=n!(n-1)!/2$ for all $n \geq 2$ while $f(1,2)=1$; and hence \cref{cycleindexm} and \cref{cycleindexm2} simplify to the following identities.
\begin{equation}{\label{cycleindexm3}}
      \sum_{n \geq 0} \sum_{a_1, \ldots, a_n \geq 0} h^{(n)}_2(a_1,\ldots,a_n) t_1^{a_1}\ldots t_n^{a_n} \dfrac{x^n}{(n!)^2} = \exp \left( t_1x + \dfrac{1}{2}\sum_{n \geq 2} \dfrac{t_nx^n}{n}\right), \;\;
      \sum_{n \geq 0} H(n,2) \dfrac{x^n}{(n!)^2} = \sqrt{\dfrac{e^x}{1-x}}.
 \end{equation}
The last identity show that $\sum_{n \geq 0} H(n,2)x^n/(n!)^2$ is not only defined as a formal power series but also has positive radius of convergence, in particular it converges for $x \in (-1,1)$. Unfortunately, similar things doesn't happen for $r \geq 3$. In other words, not only there is not exact values known for $f(n,r)$ for large $n$ and $r \geq 3$, the power series  $\sum_{n \geq 0} H(n,r)x^n/(n!)^2$ has radius of convergence $0$. To understand why this is true, we shall need an asymptotic formula for $H(n,r)$. The following asymptotic formula for $H(n,r)$ is a corollary of \cite[Theorem 1.3]{canfield2}. For any $r \geq 2$,
\begin{equation}{\label{asymp}}
 H(n,r) = \dfrac{(nr)!}{(r!)^{2n}}\exp \left(\dfrac{(r-1)^2}{2} + O(1/n) \right) = \dfrac{\sqrt{2\pi} r^{nr+1/2}n^{nr+1/2}e^{-nr+(r-1)^2/2}}{(r!)^{2n}}\left( 1+ O(1/n)\right), \; \text{ as } n \to \infty.
\end{equation}
From the above expression it is clear that for any $r \geq 3$ and $c>0$, eventually  $H(n,r)/(n!)^2 > e^{cn}$ and hence the previous observation that the power series $\sum_{n \geq 0} H(n,r)x^n/(n!)^2$ has zero radius of convergence follows. On the other hand, the asymptotics in \Cref{asymp} and Stirling's approximation allows us to conclude that $H(n,2)/(n!)^2 \sim \sqrt{e}/\sqrt{\pi n}$ as $n \to \infty$, an observation that will be used later on. This change in the behaviour of the power series from the case $r=2$ to the case $r \geq 3$ already suggest significant difference in asymptotic component structure of the respective magical squares. Using \Cref{cycleindexm3}, we can write the following result which is an analogue of the first assertion in \Cref{cycleindex3} in our context for the case $r=2$.   

\begin{proposition}
Fix $x \in (0,1)$ and consider the  probability measure $\mathbb{P}_x$ defined on $\mathbb{N} \cup \left\{0\right\}$ as follows.
$$ \mathbb{P}_x(n) := (1-x)^{1/2}e^{-x/2} H(n,2)x^n/(n!)^2, \; \forall \; n \geq 0.$$
Consider the following sampling scheme. Let $N_x \sim \mathbb{P}_x$ and $\mathcal{A} \rvert (N_x=n) \sim \text{Uniform}(\mathcal{M}(n,r))$. Then $\left\{\chi_i(\mathcal{A}) : n \geq 1\right\}$ is a collection of independent Poisson random variables with $\chi_1(\mathcal{A})  \sim \text{Poisson}(x)$ and $\chi_i(\mathcal{A})  \sim \text{Poisson}(x^i/(2i))$, for all $i \geq 2$. 
\end{proposition}

\begin{proof}
It is enough to show that 
\begin{equation} {\label{eq1}}
\mathbb{E}_x \prod_{i \geq 1} t_i^{\chi_i(\mathcal{A})} = \exp \left[ x(t_1-1) + \sum_{n \geq 2} (t_n-1)\dfrac{x^n}{2n}\right], \; \forall \; t_1,t_2, \ldots \in (0,1],
\end{equation}
where the notation $\mathbb{E}_x$ makes the dependence of the expectation on $x$ explicit. The left hand side of the above equation can be expanded into the following. 
\begin{align}
\mathbb{E}_x \prod_{i \geq 1} t_i^{\chi_i(\mathcal{A})} = \sum_{n \geq 0} \mathbb{P}_x(n) \mathbb{E}_n \prod_{i \geq 1} t_i^{\chi_i(\mathcal{A})} &= \sum_{n \geq 0} (1-x)^{1/2}e^{-x/2} \dfrac{H(n,2)x^n}{(n!)^2} \dfrac{1}{H_{n,2}} \sum_{A \in \mathcal{M}(n,2)} \prod_{i \geq 1} t_i^{\chi_i(A)} \nonumber \\
& = (1-x)^{1/2}e^{-x/2} \sum_{n \geq 0}  \dfrac{x^n}{(n!)^2}  \sum_{a_1, \ldots, a_n \geq 0} h_2^{(n)}(a_1, \ldots,a_n) \prod_{i=1}^n  t_i^{a_i} \nonumber \\
& = (1-x)^{1/2}e^{-x/2} \exp \left( t_1x + \dfrac{1}{2}\sum_{n \geq 2} \dfrac{t_nx^n}{n}\right), \label{eq2}
\end{align}
where $\mathbb{E}_n$ denotes that expectation with respect to the uniform distribution on $\mathcal{M}(n,2)$ and we have used \Cref{cycleindexm3} to derive the last equality. Expanding $\log(1-x)$ into a power series we can now show that the expression in  \Cref{eq2} is indeed the right hand side of \Cref{eq1} and this completes the proof.
\end{proof}

From the expression \Cref{cycindexm}, it is easy to see that for $\mathcal{A}_n \sim \mathcal{M}(n,r)$, we have
$$ \left( \chi_1(\mathcal{A}_n), \ldots, \chi_n(\mathcal{A}_n) \right) \stackrel{d}{=} (Z_1, \ldots, Z_n) \bigg \rvert \left( \sum_{i=1}^n iZ_i=n \right),$$
where $Z_i \sim \text{Poisson}(f(i,r)x^i/(i!)^2)$; where $x \in (0, \infty)$ is fixed. This observation tempts us to use the existing literature on Logarithmic class, as defined in \Cref{log}, discussed in marvellous detail and scope in \cite{arratiabook}. In regard to the second condition in \Cref{log} in our case, we observe that $ i\mathbb{E}Z_i = \dfrac{if(i,r)x^i}{(i!)^2}. $ For $r=2$, we have $i\mathbb{E}Z_i = x^i/2$ and hence it satisfies the conditions of \Cref{log} with $\theta =1/2$ and for the choice $x=1$. The rich theory developed in \cite{arratiabook} can be hence applied to yield the results for $r \geq 2$. Instead of using such complicated results, we shall prove them using only the power series expression in \Cref{cycleindexm} in this section. In the next section we shall demonstrate that the results for $r=2$ case can be obtained only using the relevant results for random permutations. As we shall see later for the case $r \geq 3$, we have $i\mathbb{Z}_i \to \infty$ and hence the theory developed in \cite{arratiabook} does not apply.

Further computations using the power series expressions presented in  \Cref{cycleindexm} and \Cref{cycleindexm2} enable us to write the following analogue for the second assertion of \Cref{cycleindex3} in context of uniform magical squares.

\begin{theorem}{\label{ascomp}}
Let $r \geq 2$ and $\mathcal{A}_n \sim \text{Uniform}(\mathcal{M}(n,r))$. Then 
$$ \left( \chi_1(\mathcal{A}_n), \chi_2(\mathcal{A}_n), \ldots\right) \stackrel{d}{\longrightarrow} (Z_1,Z_2,\ldots), \; \text{ as } n \to \infty,$$
where 
\begin{enumerate}
\item for the case $r=2$, we have $\left\{Z_i : i \geq 1\right\}$ to be a collection of independent poisson variables with $Z_1 \sim \text{Poisson}(1)$ and $Z_i \sim \text{Poisson}(1/(2i))$, for all $i \geq 2$.
\item for the case $r \geq 3$, we have $Z_i \equiv 0$ for all $i \geq 1$.  
\end{enumerate}
\end{theorem}

\begin{proof}
Fix $k \geq 1$. We shall show convergence of fractional moments of all order for the random vector $(\chi_1(\mathcal{A}_n), \ldots, \chi_k(\mathcal{A}_n))$ to the fractional moments of corresponding limiting distribution. We start with introducing the following notation.
$$ C_n(t_1, \ldots,t_k) := \mathbb{E}_n \prod_{i=1}^k t_i^{\chi(\mathcal{A}_n)},$$
for any $t_1, \ldots,t_k \in (0,1) $. Here $\mathbb{E}_n$ refers to the expectation with respect to the uniform distribution on $\mathcal{M}(n,r)$. Observe that 
$$ \partial_{t_1}^{r_1}\ldots \partial_{t_k}^{r_k} C_n(t_1, \ldots, t_k) \rvert_{t_i=s_i} \uparrow \mathbb{E}_n \prod_{i=1}^k (\chi_i(\mathcal{A}_n))_{(r_i)}, \; \text{ as } s_1, \ldots,s_k \uparrow 1,$$
for any $r_1, \ldots,r_k \in \mathbb{N} \cup \left\{0\right\}$ and $x_{(k)} := x(x-1)\cdots(x-k+1)$ for any $k \geq 1$ and $x_{(0)} :=1$. Plugging in $t_i=1$ for all $i >k$ in the formal power series identity in \Cref{cycleindexm} we can write the following identity.
\begin{equation}{\label{eq3}}
\sum_{n \geq 0} \dfrac{H(n,r)x^n}{(n!)^2} C_n(t_1, \ldots, t_k) = \exp \left( \sum_{n =1}^k \dfrac{f(n,r)t_nx^n}{(n!)^2} + \sum_{n >k } \dfrac{f(n,r)x^n}{(n!)^2}  \right).
\end{equation}
Differentiating formally both sides of the formal power series identity in \Cref{eq3} and comparing coefficients, we can arrive at the following conclusion.
\begin{align*}
&\partial_{t_1}^{r_1}\ldots \partial_{t_k}^{r_k} C_n(t_1, \ldots, t_k)  \rvert_{t_i=s_i} \\
&= \dfrac{(n!)^2}{H(n,r)}\textsc{Coeff}\left( \exp \left( \sum_{m =1}^k \dfrac{f(m,r)s_mx^m}{(m!)^2}+ \sum_{m >k } \dfrac{f(m,r)x^m}{(m!)^2} \right) \prod_{m=1}^k \left( \dfrac{f(m,r)x^m}{(m!)^2}\right)^{r_m} ; n\right),
\end{align*} 
for any $s_1, \ldots, s_k \in (0,1)$. Since all the co-efficients of the formal power series encountered are non-negative, we can take $s_1, \ldots,s_k \uparrow 1$ to conclude that 
\begin{align}
\mathbb{E}_n \prod_{i=1}^k (\chi_i(\mathcal{A}_n))_{(r_i)} &= \dfrac{(n!)^2}{H(n,r)}\textsc{Coeff}\left( \exp \left( \sum_{m \geq 1} \dfrac{f(m,r)x^m}{(m!)^2} \right) \prod_{m=1}^k \left( \dfrac{f(m,r)x^m}{(m!)^2}\right)^{r_m} ; n\right) \nonumber \\
& = \prod_{m=1}^k \left( \dfrac{f(m,r)}{(m!)^2}\right)^{r_m} \dfrac{(n!)^2}{H(n,r)} \textsc{Coeff}\left( \exp \left( \sum_{m \geq 1} \dfrac{f(m,r)x^m}{(m!)^2} \right) ; n - \sum_{m =1}^k mr_m\right) \nonumber \\
&= \prod_{m=1}^k \left( \dfrac{f(m,r)}{(m!)^2}\right)^{r_m} \dfrac{(n!)^2}{H(n,r)} \textsc{Coeff}\left(  \sum_{m \geq 0} \dfrac{H(m,r)x^m}{(m!)^2} ; n - \sum_{m =1}^k mr_m\right) \nonumber \\
& = \prod_{m=1}^k \left( \dfrac{f(m,r)}{(m!)^2}\right)^{r_m} \dfrac{(n!)^2}{H(n,r)} \dfrac{H(n-\sum_{m=1}^k mr_m)}{(n-\sum_{m=1}^k mr_m)!^2}, \label{frac}
\end{align}
for any $n \geq \sum_{m=1}^k mr_m.$ A brief routine computation using \Cref{asymp} and Stirling's approximation guarantee that for any $r \geq 2$ and $m \in \mathbb{N}$, we have
$$ \dfrac{H(n,r)}{(n!)^2} \sim \dfrac{H(n-m,r)}{((n-m)!)^2} \dfrac{r^{mr}}{(r!)^{2m}} n^{m(r-2)}, \; \text{ as } n \to \infty.$$
Plugging in the above mentioned asymptotic expression into \Cref{frac} and recalling that the values for $f(n,2)$ for any $n \geq 1$, we can conclude the following.
\begin{equation}{\label{res1}}
\lim_{n \to \infty} \mathbb{E}_n \prod_{i=1}^k (\chi_i(\mathcal{A}_n))_{(r_i)} = \begin{cases}
\prod_{m=2}^k (2m)^{-r_m}, \text{ if } r=2, \\
0, \text{if } r \geq 3.
\end{cases}
\end{equation}
Recalling the formula for fractional moments of Poisson random variables, we can conclude the proof. 
\end{proof}

From \Cref{ascomp}, one can easily conclude the asymptotic distribution for the size of the smallest component in case of $r =2$ as established below.

\begin{corollary}{\label{assmall2}}
Let $\mathcal{A}_n \sim \text{Uniform}(\mathcal{M}(n,2))$. Then 
$$ S_n := \inf \left\{k \geq 1 : \chi_k(\mathcal{A}_n) >0 \right\} \stackrel{d}{\longrightarrow} S_{\infty}, \; \text{with} \; \mathbb{P}(S_{\infty}=r) = \begin{cases}
1-e^{-1},& \; \text{ if } r=1, \\
\exp \left(-\dfrac{1}{2}-\dfrac{H_{r-1}}{2} \right)\left( 1- e^{-1/(2r)}\right), & \; \text{ if } r >1,
\end{cases} $$
where $H_r := \sum_{k =1}^r k^{-1}$, for all $r \geq 1$. 
\end{corollary}

\begin{proof}
We continue with the notations introduced in the statement of \Cref{ascomp}. For any $r \geq 1$, we have the following from \Cref{assmall2}.
\begin{align*}
\mathbb{P}(S_n > r) = \mathbb{P} \left( \chi_i(\mathcal{A}_n) =0, \; \forall \; i=1, \ldots, r\right) 
& \stackrel{n \to \infty}{\longrightarrow} \mathbb{P} \left( Z_i =0, \; \forall \; i=1, \ldots, r\right) \\
& = \exp\left( -1-\sum_{i=2}^r (2i)^{-1}\right) = \exp\left(-\dfrac{1}{2}-\dfrac{H_r}{2} \right).
\end{align*}
Note that, $S_{\infty}$ takes value  in $\mathbb{N}$ with probability $1$ since $H_r \to \infty$ as $r \to \infty$. 
\end{proof}

From \Cref{ascomp}, it is clear that for the case $r \geq 3$, the random element $\mathcal{A}_n \sim \text{Uniform}(\mathcal{A}(n,r))$ has no component of size $O(1)$ with high probability. In fact we shall show that the uniform element from $\mathcal{M}(n,r)$ is typically irreducible for large $n$. 

\begin{theorem}{\label{rgeq3}}
Fix $r \geq 3$ and consider $\mathcal{A}_n \sim \text{Uniform}(\mathcal{M}(n,r))$. Then $\mathbb{P}(S_n=n) \longrightarrow 1$ as $n \to \infty$. 
\end{theorem}

\begin{proof}
For any $1 \leq i \leq n$, we plug-in $r_i=1$ and $r_j =0$ for all $j \neq i$ in \Cref{frac} to obtain the following.
$$ \mathbb{E}_n \chi_i(\mathcal{A}_n) = \dfrac{f(i,r)}{(i!)^2} \dfrac{H(n-i,r)}{(n-i)!^2} \dfrac{n!^2}{H(n,r)} \leq  \dfrac{H(i,r)}{(i!)^2} \dfrac{H(n-i,r)}{(n-i)!^2} \dfrac{n!^2}{H(n,r)}, \; \forall \; 1 \leq i \leq n.$$
Fix $\varepsilon >0$. The asymptotic formula in \Cref{asymp} and Stirling's approximation yields that 
$$  \dfrac{H(m,r)}{m!^2} \sim  \dfrac{\sqrt{2\pi} r^{mr+1/2}m^{mr+1/2}e^{-mr+(r-1)^2/2}}{(r!)^{2m}2\pi m^{2m+1}e^{-2m}} = \dfrac{r^{mr+1/2}m^{m(r-2)-1/2}e^{(r-1)^2/2}}{\sqrt{2\pi}(r!)^{2m}e^{m(r-2)}}, \; \text{ as } m \to \infty.$$ 
A small computation now shows that there exists $K_{\varepsilon} \in \mathbb{N}$ such that for any $n$ and $i \in  [n]$, satisfying $i,n-i \geq K_{\varepsilon}$, we have the following.
\begin{align*}
 \mathbb{E}_n \chi_i(\mathcal{A}_n)  \leq  \dfrac{H(i,r)}{(i!)^2} \dfrac{H(n-i,r)}{(n-i)!^2} \dfrac{n!^2}{H(n,r)} &\leq (1+\varepsilon) \dfrac{\sqrt{r}i^{i(r-2)-1/2}(n-i)^{(n-i)(r-2)-1/2}e^{(r-1)^2/2}}{\sqrt{2\pi}n^{n(r-2)-1/2}}  \\
 & = (1+\varepsilon) \dfrac{\sqrt{r}e^{(r-1)^2/2}}{\sqrt{2\pi}}\left( \dfrac{i}{n}\right)^{i(r-2)} \left( 1- \dfrac{i}{n}\right)^{(n-i)(r-2)} \sqrt{\dfrac{n}{i(n-i)}} \\
 & \leq (1+\varepsilon) \dfrac{\sqrt{r}e^{(r-1)^2/2}}{\sqrt{\pi}}\left( \dfrac{i}{n}\right)^{i(r-2)} \left( 1- \dfrac{i}{n}\right)^{(n-i)(r-2)} \\
 & \leq (1+\varepsilon) \dfrac{\sqrt{r}e^{(r-1)^2/2}}{\sqrt{\pi}}\left( \dfrac{i}{n} \vee \left(1- \dfrac{i}{n} \right) \right)^{n(r-2)}. 
\end{align*}
Therefore, for any $i$ satisfying $K_{\varepsilon} \leq i \leq n/2$, we have 
$$  \mathbb{E}_n \chi_i(\mathcal{A}_n)  \leq (1+\varepsilon) \dfrac{\sqrt{r}e^{(r-1)^2/2}}{\sqrt{\pi}} \left(1- \dfrac{i}{n} \right)^{n(r-2)} \leq (1+\varepsilon) \dfrac{\sqrt{r}e^{(r-1)^2/2}}{\sqrt{\pi}}  \exp(-i(r-2)). $$
For any $l \geq K_{\varepsilon}$, the above estimate yields the following.
\begin{align*}
\mathbb{P}\left( S_n \leq n/2\right) = \mathbb{P} \left( \sum_{i \geq 1 : i \leq n/2} \chi_i(\mathcal{A}_n) \geq 1\right) & \leq \mathbb{E} \left( \sum_{i \geq 1 : i \leq n/2} \chi_i(\mathcal{A}_n)\right) \\
& \leq  \sum_{i=1}^{l} \mathbb{E}_n \chi_i(\mathcal{A}_n) + (1+\varepsilon) \dfrac{\sqrt{r}e^{(r-1)^2/2}}{\sqrt{\pi}} \sum_{i \geq l} \exp(-i(r-2)) \\
& \leq \sum_{i=1}^{l} \mathbb{E}_n \chi_i(\mathcal{A}_n) + (1+\varepsilon) \dfrac{\sqrt{r}e^{(r-1)^2/2}}{\sqrt{\pi}(1-\exp(-(r-2)))} \exp(-l(r-2)). 
\end{align*}
By $\Cref{res1}$, we have that $\mathbb{E}_n \chi_i(\mathcal{A}_n) =o(1)$, for any fixed $i$ and therefore
$$ \limsup_{n \to \infty} \mathbb{P}\left( S_n \leq n/2\right) \leq (1+\varepsilon) \dfrac{\sqrt{r}e^{(r-1)^2/2}}{\sqrt{\pi}(1-\exp(-(r-2)))} \exp(-l(r-2)).$$
Since $l$ can be taken arbitrarily large, we conclude that $\mathbb{P}\left( S_n \leq n/2\right) = o(1)$. But we observe that either the smallest component size of $\mathcal{A}_n$ is at most $n/2$ or the matrix $\mathcal{A}_n$ is irreducible, i.e., the smallest component is of size $n$. This completes the proof.  
\end{proof}

\begin{remark}
Since $\mathbb{P}(S_n=n) = f(n,r)/H(n,r)$, \Cref{rgeq3} guarantees that $f(n,r)\sim H(n,r)$ as $n \to \infty$ for $r \geq 3$. 
\end{remark}

The next statistic that we shall consider is the number of components in the matrix $\mathcal{A}_n \sim \text{Uniform}(\mathcal{M}(n,r))$, i.e., $ C_n = C(\mathcal{A}_n):= \sum_{i=1}^n \chi_i(\mathcal{A}_n)$. It is obvious that if the matrix $\mathcal{A}_n$ is irreducible, i.e., $S_n=n$, we have $C_n=1$ and hence \Cref{rgeq3} guarantees that $C_n$ converges in probability to $1$ for $r \geq 3$. For the case $r=2$, we shall be able to write down a central limit theorem for $C_n$. To guess the scaling and centering for $C_n$, observe that 
$$ \mathbb{E}C_n = \sum_{i=1}^n \mathbb{E}\chi_i(\mathcal{A}_n) \approx 1+ \sum_{i=2}^n \dfrac{1}{2i} \approx (\log n)/2,$$
and similarly 
$$\operatorname{Var}(C_n) \approx \sum_{i=1}^n \operatorname{Var}(\chi_i(\mathcal{A}_n)) \approx 1+ \sum_{i=2}^n \dfrac{1}{2i} \approx (\log n)/2.$$ The following CLT says that typically a matrix from $\mathcal{M}(n,2)$ has $(\log n)/2 \pm \sqrt{2\log n}$ many irreducible components.

\begin{theorem}
Let $\mathcal{A}_n \sim \text{Uniform}(\mathcal{M}(n,r))$ and $ C_n = C(\mathcal{A}_n):= \sum_{i=1}^n \chi_i(\mathcal{A}_n)$. Then 
$$ \dfrac{C_n - \dfrac{1}{2}\log n}{\sqrt{\dfrac{1}{2}\log n}} \stackrel{d}{\longrightarrow} N(0,1).$$
\end{theorem}

\begin{proof}
We start by plugging in $t_k \equiv 1$ for all $k \geq 1$ in the power series expression \Cref{cycleindexm3} and obtaining 
$$ \sum_{n \geq 0} \dfrac{H(n,2)}{n!^2} \mathbb{E} t^{C(\mathcal{A}_n)} x^n = \exp(tx/2)(1-x)^{-t/2}, \; \forall \; |x|<1.$$
Applying Leibniz rule and comparing co-efficients we obtain the following.
$$ \dfrac{H(n,2)}{n!^2} \mathbb{E} t^{C_n} = \dfrac{1}{n!}\sum_{k =0}^n {n \choose k} \dfrac{\partial^{n-k}}{\partial x^{n-k}}\exp(tx/2) \big \rvert_{x=0} \dfrac{\partial^k}{\partial x^{k}}(1-x)^{-t/2} \big \rvert_{x=0}  = \dfrac{1}{n!}\sum_{k =0}^n {n \choose k} (t/2)^{n-k} \dfrac{\Gamma(t/2+k)}{\Gamma(t/2)}.$$
Fix some sequence $t_n \to 1$ as $n \to \infty$ and define
$$ u_{n,k} = \begin{cases}
\dfrac{n!^2}{H(n,2) \Gamma(t_n/2)}  \dfrac{t_n^k \Gamma(t_n/2+n-k)}{2^k(n-k)!k!}, \; & \text{ if } n \geq k, \\
0, & \text{ otherwise }.
\end{cases}$$
Clearly, $\mathbb{E}t_n^{C_n} = \sum_{k \geq 0} u_{n,k}.$ We now apply Stirling's approximation to compute the limit of $u_{n,k}$ as $n \to \infty$ for any fixed $k \geq 0$. 
\begin{align*}
u_{n,k} &\sim \dfrac{n!^2}{H(n,2) \Gamma(t_n/2)}  \dfrac{t_n^k \Gamma(t_n/2+n-k)}{2^k(n-k)!k!} \\
& \sim \dfrac{1}{2^kk!} \dfrac{\sqrt{\pi n}}{\sqrt{e}\Gamma(1/2)} \dfrac{e^{n-k}}{\sqrt{2\pi} (n-k)^{n-k+1/2}} \dfrac{\sqrt{2\pi}(t_n/2+n-k)^{t_n/2+n-k-1/2}}{\exp(t_n/2+n-k)} \\
& \sim  \dfrac{1}{2^kk!} \dfrac{\sqrt{ n}}{\sqrt{e}} \dfrac{e^{n-k}}{\sqrt{2\pi} e^{-k}n^{n-k+1/2}} \dfrac{\sqrt{2\pi}e^{t_n/2-k}n^{t_n/2+n-k-1/2}}{\exp(t_n/2+n-k)} \\
& =  \dfrac{1}{2^kk!} \dfrac{n^{(t_n-1)/2}}{\sqrt{e}}.
\end{align*}
We want to apply DCT and conclude that 
$$ n^{(1-t_n)/2}\mathbb{E}t_n^{C_n} = n^{(1-t_n)/2} \sum_{k \geq 0} u_{n,k} \longrightarrow \sum_{k \geq 0} \dfrac{1}{2^kk!} \dfrac{1}{\sqrt{e}}=1.$$
To verify validity of DCT, we can assume without loss of generality that $1/2 \leq t_n \leq 2$ for all $n \geq 1$ and hence there exists constants $C_1,C_2 < \infty$ and $C_3,C_4 >0$ such that 
$$ \Gamma(x) \leq C_1 \sqrt{2\pi} x^{x-1/2}e^{-x}, \; \forall \; x \geq 1/4, \; (n!)^2/(H(n,2)) \leq C_2\sqrt{n}, \; \forall \; n \geq 1,$$
and 
$$ m! \geq C_3 \sqrt{2\pi} m^{m+1/2}e^{-m}, \;\forall \; m \geq 1, \; \Gamma(x) \geq C_4, \; x \geq 1/4.$$ 
Plugging-in these estimates we have
\begin{align*}
n^{(1-t_n)/2}u_{n,k} &\leq n^{(1-t_n)/2}\dfrac{1}{k!}\dfrac{e^{n-k}C_2 \sqrt{n}}{C_3 \sqrt{2\pi} (n-k)^{n-k+1/2}C_4}C_1\sqrt{2\pi} (t_n/2+n-k)^{t_n/2+n-k-1/2}e^{-t_n/2-n+k} \\
& \leq C_5\dfrac{1}{k!}  \dfrac{1}{\sqrt{n-k}} (t_n/2+n-k)^{-1/2} n^{1-t_n/2} e^{-t_n/2}(t_n/2+n-k)^{t_n/2} \left( 1+ \dfrac{t_n}{2(n-k)}\right)^{n-k} \\
& \leq C_5\dfrac{1}{k!}  \dfrac{1}{\sqrt{n-k}} (t_n/2+n-k)^{-1/2} n^{1-t_n/2} e^{-t_n/2}(t_n/2+n-k)^{t_n/2} e^{t_n/2} \; ( \text{since } (1+y) \leq e^y) \\
& \leq C_5 \dfrac{1}{k!} \sqrt{\dfrac{n}{n-k}}  \sqrt{\dfrac{n}{t_n/2+n-k}} \left( 1 + \dfrac{t_n/2-k}{n}\right)^{-t_n/2} \\
& \leq C_5 \dfrac{1}{k!} \dfrac{n}{n-k}   \left( 1 + \dfrac{1/4-k}{n}\right)^{-1} \leq C_5 \dfrac{1}{k!} \dfrac{n}{n-k}   \left( 1 - k/n\right)^{-1} \leq C_5 \dfrac{1}{k!} \left( 1 - k/n\right)^{-2} \leq \dfrac{C_5 (k+1)^2}{k!},
\end{align*}
for all $n \geq k+1$. On the otherhand,
$$ k^{(1-t_k)/2}u_{k,k} =\dfrac{t_k^k}{2^kk!} \dfrac{k!^2}{H(k,2)} k^{(1-t_k)/2} \leq \dfrac{C_2}{k!}k^{3/4}.$$
This shows that $n^{(1-t_n)/2}u_{n,k} \leq C_6k^2/(k!)$ for all $n \geq 1$ and the bound is summable in $k$, proving the validity of DCT. Fix $s \in \mathbb{R}$ and set $t_n = \exp(\sigma_n^{-1}s)$, where $\sigma_n^2 = \log n$ and $\mu_n := (\log n)/2$. Then 
\begin{align*}
\log \mathbb{E} \exp(s\sigma_n^{-1}(C_n-\mu_n)) = \log \mathbb{E} t_n^{C_n} - s\sigma_n^{-1} \mu_n &= \dfrac{1}{2}(t_n-1) (\log n) - s\sigma_n^{-1} \mu_n +o(1) \\
& = \dfrac{1}{2}\log n \left( t_n-1- \dfrac{s}{\sigma_n}\right) + o(1) \\
& = -\dfrac{1}{2}(\log n) \left( \log t_n +1 -t_n \right) + o(1) \\
&= -\dfrac{1}{2}(\log n) \left( -\dfrac{1}{2}(1-t_n)^2\right)(1+o(1))+ o(1) \\
& = \dfrac{1}{4}(\log n)(1-\exp(s/\sigma_n))^2(1+o(1)) +o(1) \\
& = \dfrac{1}{4}(\log n)s^2\sigma_n^{-2}(1+o(1)) +o(1) \to s^2/4.
\end{align*}
This proves the theorem.
\end{proof}

\section{Computation Using Symmetric Group}
\label{symm}

In this section we shall discuss a connection between symmetric group pf order $n$, denoted by $\mathscr{S}_n$, and the set of all $2$-regular bipartite graphs on $2n$ vertices, denoted by $\mathcal{M}(n,2)$. In particular, we shall derive an importance sampling scheme which helps us to estimate averages of statistics of the uniform random elements from $\mathcal{M}(n,2)$ based on a random sample of uniform permutations of order $n$. Moreover, we shall show how some aspects of component spectrum of uniform element from $\mathcal{M}(n,2)$ can be derived using only the results for cycle structure of random permutations.

With the above mentioned goal in mind, we define $\mathcal{M}^*(n,2)$ to be the set of all edge-coloured (with Blue and Red)  $2$-regular bipartite graphs on $2n$ vertices, i.e., each vertex has exactly one blue and one red edge adjacent to it. The following easy lemma counts the cardinality of $\mathcal{M}^*(n,2)$.
\begin{lemma}{\label{card}}
$\operatorname{card}\left( \mathcal{M}^*(n,2)\right) = (n!)^2, \; \forall \; n \geq 1.$
\end{lemma}

\begin{proof}
Recall the notational convention where we label the two bi-partitions of a regular bipartite graph with $2n$ vertices as $\left\{R_1, \ldots, R_n\right\}$ and $\left\{C_1, \ldots, C_n\right\}$. Fix $G \in \mathcal{M}^*(n,2)$; let $\pi_{\texttt{blue}}^G(i)=j$ if there is a blue edge connecting $R_i$ and $C_j$ in $G$ and $\pi_{\texttt{red}}^G(i)=j$ if there is a red edge connecting $R_i$ to $C_j$ in $G$. It is easy to see that the map $G \mapsto (\pi_{\texttt{blue}}^G,\pi_{\texttt{red}}^G)$ is a bijection from $\mathcal{M}^*(n,2)$ to $\mathscr{S}_n \times \mathscr{S}_n$ and hence the lemma follows.
\end{proof}

The cycle structures of the two permutations defined in the proof of \Cref{card}, $\pi_{\texttt{blue}}^G,\pi_{\texttt{red}}^G$, also capture the component structure of the graph $G$. Note that  the vertices $R_i$ and $R_j$ are connected in the graph $G$, i.e., they are in the same component, if and only if there exists $k \in [n]$, distinct $u_0=i,u_1, \ldots, u_{k}=j \in [n]$ and distinct $v_1, \ldots, v_k \in [n]$ such that there exists an edge between $R_{u_l}$ and $C_{v_{l+1}}$ for all $l=0, \ldots,k-1$ and an edge between $C_{v_l}$ and $R_{v_l}$ for all $l=1, \ldots,k$. This happens if and only if one of the following two situations occur : either $\pi_{\texttt{blue}}^G(u_l)=v_{l+1}$ for $l=0,\ldots,k-1$ and $\pi_{\texttt{red}}^G(u_l)=v_{l}$ for $l=1,\ldots,k$, or $\pi_{\texttt{red}}^G(u_l)=v_{l+1}$ for $l=0,\ldots,k-1$ and $\pi_{\texttt{blue}}^G(u_l)=v_{l}$ for $l=1,\ldots,k$. This is equivalent to saying that $i$ and $j$ are in the same cycle of the permutation $\pi_{\texttt{red}}^G \circ (\pi_{\texttt{blue}}^G)^{-1}$. Therefore, the graph $G$ has exactly $a_k$ many components of size $k$ if and only if the permutation $\pi_{\texttt{red}}^G \circ (\pi_{\texttt{blue}}^G)^{-1}$ has exactly $a_k$ many $k$-cycles. The following lemma encapsulates this discussion in terms of distribution of component spectrum of an uniform random element from $\mathcal{M}^*(n,2)$.

\begin{lemma}{\label{card3}}
Let $\mathcal{G} \sim \text{Uniform} \left(\mathcal{M}^*(n,2) \right)$. Then $\pi_{\texttt{red}}^\mathcal{G}, \pi_{\texttt{blue}}^\mathcal{G} \stackrel{ind}{\sim} \text{Uniform}(\mathscr{P}_n)$ and hence $\pi_{\texttt{red}}^\mathcal{G} \circ (\pi_{\texttt{blue}}^\mathcal{G})^{-1} \sim  \text{Uniform}(\mathscr{P}_n)$. In particular, 
$$ \left( \chi_1(\mathcal{G}), \ldots, \chi_n(\mathcal{G})\right) \stackrel{d}{=} \left( c_1(\Pi), \ldots, c_n(\Pi)\right),$$
where $\Pi \sim \text{Uniform}(\mathscr{P}_n)$.
\end{lemma}

\begin{proof}
The proof is trivial from the fact that the map $G \mapsto (\pi_{\texttt{blue}}^G,\pi_{\texttt{red}}^G)$ is a bijection.
\end{proof}

We now want to return to the discussion of $\mathcal{M}(n,2)$. Any graph $G$ of $\mathcal{M}^*(n,2)$ maps to, after deleting the edge-colouring information,  an element $\psi(G)$ of $\mathcal{M}(n,2)$. Unfortunately, this map $\psi$ is not injective; otherwise we could have trivially reduced the discussion of the component spectrum of typical element in $\mathcal{M}(n,2)$ to the discussion of cycle structure of typical permutations. The following short lemma gives us an idea about the degree of non-injectivity of the map $\psi$.

\begin{lemma}{\label{card2}}
For any $A \in \mathcal{M}(n,2)$, we have $\operatorname{card}\left( \psi^{-1}\left(\left\{A\right\} \right)\right) = 2^{\sum_{i=2}^n \chi_i(A)}. $
\end{lemma} 

\begin{proof}
We start with $A \in \mathcal{M}(n,2)$ and try to assign edge-colours to $A$. Consider any irreducible $r$-component $B_r$ of $A$, i.e., component with $2r$ vertices. It is easy to see that once we assign colour to one of the edges in $B$, colouring of the other edges in $B$ are determined; hence there are at most $2$ ways to colour any connected $r$-component. For $r=1$, these two colourings are equivalent whereas it is not the case for $r \geq 2$. This proves the lemma.
\end{proof}

Now let $\mathbb{Q}_{\texttt{uniform},n}$ denotes the uniform probability measure on $\mathcal{M}(n,2)$ while  $\mathbb{Q}_{\texttt{col},n}$ denotes the probability measure induced on $\mathcal{M}(n,2)$ by $\psi(\mathcal{G}_n)$ where $\mathcal{G}_n \sim \text{Uniform} \left(\mathcal{M}^*(n,2) \right)$. For any $A \in \mathcal{M}(n,2)$, we have
$$ \mathbb{Q}_{\texttt{col},n}(\left\{A\right\}) = \mathbb{P}\left(\psi(\mathcal{G}_n)=A \right) = \dfrac{\operatorname{card}\left( \psi^{-1}\left(\left\{A\right\} \right)\right)}{\operatorname{card}\left( \mathcal{M}^*(n,2)\right)} = \dfrac{1}{n!^2}2^{C(A)-\chi_1(A)}, $$
and hence
\begin{equation}{\label{simul}}
\dfrac{d\mathbb{Q}_{\texttt{uniform},n}}{d\mathbb{Q}_{\texttt{col},n}} (A) = \dfrac{2^{-C(A)+\chi_1(A)}}{H(n,2)/(n!^2)}, \; \forall \; A \in \mathcal{M}(n,2).
\end{equation}
The formula in \Cref{simul} gives us an importance sampling method to estimate $\mathbb{E}h(\mathcal{A}_n)$ for any $h :\mathcal{M}(n,2) \to \mathbb{R}$ and  $\mathcal{A}_n \sim \text{Uniform}(\mathcal{M}(n,2)$, as described in the following algorithm.

\RestyleAlgo{boxruled}
\LinesNumbered
\begin{algorithm}[ht]
  \caption{Importance Sampling for $\text{Uniform}(\mathcal{M}(n,2))$\label{alg}}
  Fix $N$ large. \\
  Generate $\Pi_{1,1}, \Pi_{2,1}, \ldots, \Pi_{1,N},\Pi_{2,N} \stackrel{ind}{\sim} \text{Uniform}(\mathscr{P}_n).$ \\
  Set $\mathcal{G}_i = \varphi^{-1}(\Pi_{1,i},\Pi_{2,i})$, for all $i \in [N]$, where $\varphi$ is the bijection $ G \mapsto (\pi_{\texttt{blue}}^G,\pi_{\texttt{red}}^G)$. Clearly, $\psi(\mathcal{G}_i ) \stackrel{ind}{\sim} \mathbb{Q}_{\texttt{col},n}.$ \\
  Estimate $\mathbb{E}h(\mathcal{A}_n)$ by 
  $$ \dfrac{1}{N}\dfrac{(n!)^2}{H(n,2)} \sum_{i=1}^N h(\psi(\mathcal{G}_i)) 2^{-C(\mathcal{G}_i)+\chi_1(\mathcal{G}_i)}.$$
\end{algorithm}
In the special case where the function $h$ depends on its argument only through its component structure, i.e., $\exists \; h^* : \mathbb{Z}_{\geq 0}^n \to \mathbb{R}$ such that $h(A) = h^*(\chi_1(A), \ldots, \chi_n(A))$, then we combine \Cref{card3} and \Cref{simul} to conclude the following result.

\begin{result}{\label{resimport}}
Let $\mathcal{A}_n \sim \text{Uniform}(\mathcal{M}(n,2))$ and $h$ be a function as described above. Then 
$$ \mathbb{E} h(\mathcal{A}_n) = \dfrac{n!^2}{H(n,2)} \mathbb{E}h^*(c_1(\Pi), \ldots, c_n(\Pi)) 2^{-C(\Pi)+c_1(\Pi)},$$
where $\Pi \sim \text{Uniform}(\mathscr{P}_n)$. In particular, plugging in $h \equiv 1$, one can conclude that 
$$ \mathbb{E} 2^{-C(\Pi)+c_1(\Pi)} = \dfrac{H(n,2)}{n!^2},$$
and hence
$$\mathbb{E} h(\mathcal{A}_n) = \dfrac{ \mathbb{E}h^*(c_1(\Pi), \ldots, c_n(\Pi)) 2^{-C(\Pi)+c_1(\Pi)}}{ \mathbb{E} 2^{-C(\Pi)+c_1(\Pi)}}.$$
\end{result} 
 
\Cref{resimport} changes the problem of studying component spectrum of an uniform element from $\mathcal{M}(n,2)$ to the analysis of uniform random permutation, a topic which is very well-understood and explored. As an warm-up, let us prove the fact that $H(n,2) \sim \sqrt{e/\pi} n!^2 n^{-1/2}$ using \Cref{resimport}. Set $\beta_n := H(n,2)/(n!)^2 \in [0,1]$ for all $n \geq 1$ and note that 
$ \beta_n = \mathbb{E} 2^{-C(\Pi_n)+c_1(\Pi_n)},$
where $\Pi_n \sim \text{Uniform}(\mathscr{P}_n)$. Applying the notations introduced in the  first assertion of \Cref{cycleindex3}, one can conclude that,
\begin{align*}
(1-x)\sum_{n \geq 0} x^n \beta_n = (1-x) \sum_{n \geq 0} x^n \mathbb{E} 2^{-C(\Pi_n)+c_1(\Pi_n)} =  \mathbb{E} 2^{-C(\Pi_x)+c_1(\Pi_x)} 
& = \prod_{i \geq 2} \mathbb{E} 2^{-\text{Poisson}(x^i/i)} \\
& = \exp \left( -\dfrac{1}{2}\sum_{i \geq 2} \dfrac{x^i}{i}\right) = \exp(x/2)\sqrt{1-x},
\end{align*} 
for all $x \in (0,1)$. Further computation shows that
\begin{align*}
\sum_{n \geq 0} x^n \beta_n = \dfrac{e^{x/2}}{\sqrt{1-x}} \Rightarrow \sum_{n \geq 0} nx^{n-1}\beta_n = \dfrac{1}{2}e^{x/2}(1-x)^{-1/2} + \dfrac{1}{2}e^{x/2}(1-x)^{-3/2} = \dfrac{1}{2}e^{x/2}(2-x)(1-x)^{-3/2},
\end{align*}
and hence
$$ \sum_{n \geq 0} x^n\left( n\beta_n - (n-1)\beta_{n-1}\right) = (1-x)\sum_{n \geq 0} nx^n \beta_n = \dfrac{1}{2}e^{x/2}x(2-x)(1-x)^{-1/2}, \; \forall \; x \in (0,1),$$
with the convention that $\beta_{-1}:=0$. Since the above power series converges on $(-1,1)$, the above expression is true for any $|x|<1$ and 
$$ \lim_{x \uparrow 1} (1-x)^{1/2}\sum_{n \geq 0} x^n\left( n\beta_n - (n-1)\beta_{n-1}\right) = \sqrt{e}/2.$$
The situation is now ripe for an application of a version of  \textit{Hardy-Littlewood Tauberian Theorem}, which can be found in \cite[Theorem 7.4]{korevaar}. Applying this theorem we can conclude that, provided there exists $C < \infty$ such that $n(n\beta_n-(n-1)\beta_{n-1}) \leq C\sqrt{n}$ for all $n \geq 1$, we have   
$$ n^{-1/2} \sum_{k =0}^n \left( k\beta_k - (k-1)\beta_{k-1}\right) \longrightarrow \dfrac{\sqrt{e}/2}{\Gamma(3/2)}.$$
The last convergence implies that $\sqrt{n}\beta_n \longrightarrow \sqrt{e/\pi}$, which completes our argument.

To prove the sufficient condition required to apply the Tauberian theorem, we recall the map $\omega_n : \mathscr{P}_n \to \mathscr{P}_{n-1}$ where $\omega_n(\pi)$ is the permutation obtained from $\pi \in \mathscr{P}_n$ after deleting $n$ from its cycle representation. For example, if $n=6$ and $\pi = (1\, 4)(2\, 6\, 5)(3)$ in its cycle representation, then $\omega_6(\pi)=(1\, 4)(2\, 5)(3)$. It is an well-known fact that $\Pi \sim \text{Uniform}(\mathscr{P}_n)$ implies that $\omega_n(\Pi) \sim \text{Uniform}(\mathscr{P}_{n-1})$. Therefore, we have
$$ n\beta_n - (n-1)\beta_{n-1} = n \mathbb{E} 2^{-C(\Pi_n)+c_1(\Pi_n)} -(n-1)\mathbb{E} 2^{-C(\omega_n(\Pi_n))+c_1(\omega_n(\Pi_n))}.$$
Introducing the notation $c_{-1}(\pi) := C(\pi)-c_1(\pi)$ to denote the number of non-fixed points of the permutation $\pi$, we observe that if  $n$ belongs to an $R_n$-cycle of $\Pi_n$, then 
$c_{-1}(\Pi_n) = c_{-1}(\omega_n(\Pi_n)) + \mathbbm{1}(R_n=2)$. Moreover, for any $r \in [n]$ and $\pi \in \mathscr{P}_{n-1}$,
\begin{align*}
\mathbb{P} \left( \omega_n(\Pi_n) = \pi, R_n =r\right) &= \dfrac{1}{n!}\operatorname{card}\left( \left\{\pi^* \in \mathscr{S}_n : \omega_n(\pi^*)=\pi, \; n \text{ is in an $r$-cycle of } \pi^* \right\}\right) \\
& = \dfrac{(r-1)c_{r-1}(\pi)}{n!}\mathbbm{1}(r >1) + \dfrac{1}{n!}\mathbbm{1}(r=1), 
\end{align*}
and therefore, $\mathbb{P}(R_n=2 \mid \omega_n(\Pi_n)) = c_1(\omega_n(\Pi_n))/n.$ Therefore,
\begin{align*}
n\beta_n - (n-1)\beta_{n-1} &= n \mathbb{E}\left(2^{-c_{-1}(\Pi_n)} - 2^{-c_{-1}(\omega_n(\Pi_n))} \right) + \mathbb{E} 2^{-c_{-1}(\omega_n(\Pi_n))} \\
& = - n \mathbb{E} 2^{-c_{-1}(\omega_n(\Pi_n))-1} \mathbbm{1}\left( R_n=2\right) + \mathbb{E} 2^{-c_{-1}(\omega_n(\Pi_n))} = \mathbb{E}2^{-c_{-1}(\omega_n(\Pi_n))}\left( 1- \dfrac{1}{2}c_1(\omega_n(\Pi_n))\right).
\end{align*}
Thus it is enough to show that 
$$ \sup_{n \geq 1} \sqrt{n}\mathbb{E}2^{-c_{-1}(\Pi_n)}\left( 1- \dfrac{1}{2}c_1(\Pi_n)\right) < \infty. $$
In order to prove the above, we plug-in $t_1=t$ and $t_k=1/2$ for all $k \geq 2$ in \Cref{cycleindex}.
$$ \sum_{n \geq 0} x^n \mathbb{E} t^{c_1(\Pi_n)}2^{-c_{-1}(\Pi_n)} = \exp \left(tx + \sum_{i \geq 2} x^{i}/(2i) \right) = e^{(t-1/2)x}(1-x)^{-1/2}, \; \forall \; |x| < 1.$$
Differentiating with respect to $t$, we obtain
\begin{equation}{\label{diff}}
 \sum_{n \geq 0} x^n \mathbb{E} c_1(\Pi_n)t^{c_1(\Pi_n)-1}2^{-c_{-1}(\Pi_n)}  = xe^{(t-1/2)x}(1-x)^{-1/2}.
\end{equation}
Think of the above differentiation as differentiation of co-efficients of formal power series as we have done multiple times earlier. Since we shall only compare co-efficients, this will be enough justification for us. Plug-in $t=1$ in the abobe two equations to obtain the following.
$$ \sum_{n \geq 0} x^n \mathbb{E}2^{-c_{-1}(\Pi_n)}\left( 1- \dfrac{1}{2}c_1(\Pi_n)\right) = e^{x/2}(1-x/2)(1-x)^{-1/2} = \left( 1- \sum_{k \geq 1} \dfrac{k-1}{2^k k!}x^k\right)\left( \sum_{k \geq 0} 2^{-2k}{2k \choose k} x^k \right).$$ 
Comparing co-efficients we obtain,
\begin{align*}
\mathbb{E}2^{-c_{-1}(\Pi_n)}\left( 1- \dfrac{1}{2}c_1(\Pi_n)\right)  = 2^{-2n}{2n \choose n} - \sum_{k =1}^n \dfrac{(k-1)2^{2(n-k)}}{2^kk!}{2(n-k) \choose (n-k)} \leq 2^{-2n}{2n \choose n} \sim \dfrac{1}{\sqrt{\pi n}},
\end{align*}
as $n \to \infty$, by Stirling's approximation. This completes the argument.

We now present a concrete example of application of \Cref{resimport} where we prove the asymptotic behaviour of the largest component size $L(\mathcal{A}_n) := \max \left\{j \geq 1 : \chi_j(\mathcal{A}_n) >0\right\}$ where $\mathcal{A}_n \sim \text{Uniform}(\mathcal{M}(n,2))$.

\begin{theorem}
Let $\mathcal{A}_n \sim \text{Uniform}(\mathcal{M}(n,2))$ and $L(\mathcal{A}_n)$ is the size of the largest irreducible component of $\mathcal{A}_n$. Then $L(\mathcal{A}_n)/n$ converges weakly to the probability measure $\mathcal{L}$ on $(0,1)$ characterized by its moments given as
$$ \int_{(0,1)} u^m \, d\mathcal{L}(u) := \dfrac{\sqrt{e}}{2\Gamma(m+1/2)} \int_{0}^{\infty} \exp \left(-\dfrac{1}{2}\int_{u}^{\infty} \dfrac{e^{-v}}{v}\, dv -u \right) u^{m-1}\, du, \; \forall \; m \geq 1.$$
\end{theorem}

\begin{proof}
By \Cref{resimport}, we have
$$ \mathbb{E} (L(\mathcal{A}_n)/n)^m = \mathbb{E} \dfrac{n!^2}{H(n,2)} L(\Pi_n)^m 2^{-C(\Pi_n)+c_1(\Pi_n)}n^{-m}, \; \forall \; m \geq 1,$$
where $\Pi_n \sim \text{Uniform}(\mathscr{P}_n)$ and $L(\Pi_n)$ is the size of the largest cycle of $\Pi_n$. 
Our goal is to use method of moments, which is valid since $L(\mathcal{A}_n)/n \in [0,1]$.  Since $H(n,2)\sim \sqrt{e/\pi}n!^2n^{-1/2}$, it is enough to try to compute the limit of $n^{-m+1/2}\mathbb{E}  L(\Pi_n)^m 2^{-c_{-1}(\Pi_n)}$ as $n \to \infty$. Recall the notations of \Cref{cycleindex3}. For any $x \in (0,1)$, we have $c_i(\Pi_x)\stackrel{ind}{\sim} \text{Poisson}(x^i/i)$ and hence the following two expressions are easy to see for any $j \geq 2$. 
\begin{align*}
\mathbb{E}\left( 2^{-c_{-1}(\Pi_x)} \mid L(\Pi_x)=j \right) & = \mathbb{E} \left( 2^{-\sum_{i=1}^j c_i(\Pi_x)} \mid c_k(\Pi_x)=0, \; \forall \; k>j, \; c_j(\Pi_x)>0\right) \\
& = \prod_{i=2}^{j-1} \mathbb{E} 2^{-\text{Poisson}(x^i/i)} \mathbb{E}\left( 2^{-\text{Poisson}(x^j/j)} \mid \text{Poisson}(x^j/j)>0\right) \\
&= \exp\left( -\sum_{i=2}^{j-1}\dfrac{x^i}{2i}\right) \dfrac{\exp(-x^j/(2j))-\exp(-x^j/j)}{1-\exp(-x^j/j)},
\end{align*}
and hence
\begin{align*}
\mathbb{E} L(\Pi_x)^m2^{-c_{-1}(\Pi_x)} \mathbbm{1}(L(\Pi_x)=j) &= j^m \mathbb{P}(L(\Pi_x)=j) \mathbb{E}\left( 2^{-c_{-1}(\Pi_x)} \mid L(\Pi_x)=j \right) \\
& = j^m \mathbb{E} \left( c_k(\Pi_x)=0, \; \forall \; k>j, \; c_j(\Pi_x)>0\right)\mathbb{E}\left( 2^{-c_{-1}(\Pi_x)} \mid L(\Pi_x)=j \right) \\
&= j^m \exp\left( -\sum_{k >j} x^k/k\right)\left( 1- \exp(-x^j/j)\right)\mathbb{E}\left( 2^{-c_{-1}(\Pi_x)} \mid L(\Pi_x)=j \right) \\
&= j^m e^{x/2}\exp\left( -\sum_{k >j} \dfrac{x^k}{k} - \sum_{k =1}^j \dfrac{x^k}{2k}\right)\left( 1- \exp(-x^j/(2j))\right).
\end{align*}
It is easy to directly check that the above expression is true for $j=1$ also and hence
\begin{equation}{\label{express}}
\mathbb{E}  L(\Pi_x)^m2^{-c_{-1}(\Pi_x)} = \sum_{j \geq 1} j^m e^{x/2}\exp\left( -\sum_{k >j} \dfrac{x^k}{k} - \sum_{k =1}^j \dfrac{x^k}{2k}\right)\left( 1- \exp(-x^j/(2j))\right)
\end{equation}
The rest of the proof is similar to the argument applied in \cite{shepp}. We define $\lambda_0(x):=0, \lambda_j(x) := \sum_{k =1}^j x^k/k,$ for $j \geq 1$ and $\lambda_{\infty}(x) := \sum_{k \geq 1} x^k/k = -\log(1-x)$, for all $x \in [0,1).$ Then \Cref{express} can be re-written as 
\begin{align}{\label{express2}}
\mathbb{E}  L(\Pi_x)^m2^{-c_{-1}(\Pi_x)} &= \sum_{j \geq 1} j^m e^{x/2}\exp\left( -\lambda_j(x)/2\right) \exp(\lambda_j(x)-\lambda_{\infty}(x))\left( 1- \exp(-(\lambda_j(x)-\lambda_{j-1}(x))/2)\right) \nonumber \\
&= \dfrac{1}{2}\sum_{j \geq 1} j^m e^{x/2} \exp(-\lambda_{\infty}(x)/2)\int_{\lambda_{j-1}(x)}^{\lambda_j(x)} \exp\left( -\dfrac{1}{2}(\lambda_{\infty}(x)-y)\right)\, dy \nonumber \\
& =  \dfrac{e^{x/2}(1-x)^{1/2}}{2}\sum_{j \geq 1} j^m   \int_{\lambda_{j-1}(x)}^{\lambda_j(x)} \exp\left( -\dfrac{1}{2}(\lambda_{\infty}(x)-y)\right)\, dy.
\end{align}
Introduce the notation $f(y):= \int_{y}^{\infty} z^{-1}e^{-z}\, dz < \infty$, for all $y>0$. Then we have the following inequalities for any $x \in [0,1)$. 
\begin{align*}
\int_{j-1}^j u^{-1}x^u\, du \geq j^{-1}x^j \geq  \int_{j}^{j+1} u^{-1}x^u\, du, \; \forall \; j \geq 1 \Rightarrow \int_{j}^{\infty} u^{-1}x^u\, du \geq \lambda_{\infty}(x)-\lambda_j(x) \geq  \int_{j+1}^{\infty} u^{-1}x^u\, du, \; \forall \; j \geq 0.
\end{align*}
Since
$$ \int_{v}^{\infty} u^{-1}x^u \, du = \int_{v}^{\infty} \dfrac{e^{u \log x}}{u}\, du = \int_{-v \log x}^{\infty} z^{-1}e^{-z}\, dz = f(-v\log x), \; \forall\; v \geq 0, x \in (0,1).$$
Therefore, $f(-j\log x) \geq \lambda_{\infty}(x)-\lambda_j(x) \geq f(-(j+1)\log x)$, for all $j \geq 0$, adopting the convention $f(0) = \infty$, since  $f$ is strictly decreasing and differentiable with $f(y) \uparrow \infty$ as $y \downarrow 0$. Hence there exists a unique $u_j(x) \in [-j\log x,-(j+1)\log x]$ such that $f(u_j(x)) = \lambda_{\infty}(x)-\lambda_j(x),$ for all $j \geq 0$. This yields the following.
\begin{align*}
\int_{\lambda_{j-1}(x)}^{\lambda_j(x)} \exp\left( -\dfrac{1}{2}(\lambda_{\infty}(x)-y)\right)\, dy &= \int_{\lambda_{\infty}(x)-\lambda_j(x)}^{\lambda_{\infty}(x)-\lambda_{j-1}(x)} e^{-y/2}\, dy \\
&=\int_{u_j(x)}^{u_{j-1}(x)} e^{-f(v)/2}f^{\prime}(v)\, dv = \int_{u_{j-1}(x)}^{u_j(x)} e^{-f(v)/2-v}v^{-1}\, dv.
\end{align*}
Plugging in this expression in \Cref{express2} we obtain the following.
\begin{align}{\label{express3}}
\mathbb{E}  L(\Pi_x)^m2^{-c_{-1}(\Pi_x)}
& =  \dfrac{e^{x/2}(1-x)^{1/2}}{2}\sum_{j \geq 1} j^m  \int_{u_{j-1}(x)}^{u_j(x)} e^{-f(v)/2-v}v^{-1}\, dv. 
\end{align}
Our goal is to derive the asymptotic behaviour of the above quantity as $x \uparrow 1$. With that goal in mind, we define three related quantities.
$$ S_L(x) :=   \dfrac{e^{x/2}}{2}\sum_{j \geq 1} (u_{j-1}(x))^m  \int_{u_{j-1}(x)}^{u_j(x)} e^{-f(v)/2-v}v^{-1}\, dv,$$
$$ S_M(x) :=   \dfrac{e^{x/2}}{2}\sum_{j \geq 1}   \int_{u_{j-1}(x)}^{u_j(x)} e^{-f(v)/2-v}v^{m-1}\, dv = \dfrac{e^{x/2}}{2}  \int_{u_{0}(x)}^{\infty} e^{-f(v)/2-v}v^{m-1}\, dv,$$
$$ S_U(x) :=   \dfrac{e^{x/2}}{2}\sum_{j \geq 1} (u_{j}(x))^m  \int_{u_{j-1}(x)}^{u_j(x)} e^{-f(v)/2-v}v^{-1}\, dv.$$
Since $u_{j-1}(x) \leq -j\log x \leq u_j(x)$ It is straightforward to see that 
$$ S_L(x) \leq S_M(x), (-\log x)^m (1-x)^{-1/2}\mathbb{E}  L(\Pi_x)^m2^{-c_{-1}(\Pi_x)} \leq S_U(x), \; \forall \; x \in [0,1).$$
It is easy to see that $u_j(x)-u_{j-1}(x) \leq -(j+1)\log x +(j-1)\log x = -2\log x$ and hence
\begin{align*}
S_U(x)-S_L(x) &=  \dfrac{e^{x/2}}{2}\sum_{j \geq 1} \left((u_{j}(x))^m - (u_{j-1}(x))^m \right)  \int_{u_{j-1}(x)}^{u_j(x)} e^{-f(v)/2-v}v^{-1}\, dv \\
& \leq \dfrac{e^{x/2}}{2}\sum_{j \geq 1} m(u_j(x))^{m-1}\left(u_{j}(x) - u_{j-1}(x) \right)  \int_{u_{j-1}(x)}^{u_j(x)} e^{-f(v)/2-v}v^{-1}\, dv \\
& \leq -m(\log x) e^{x/2}\sum_{j \geq 1} (u_j(x))^{m-1} \int_{u_{j-1}(x)}^{u_j(x)} e^{-f(v)/2-v}v^{-1}\, dv \\
& \leq -m(\log x) e^{x/2}\sum_{j \geq 1} 2^{m-1}\left( (u_{j-1}(x))^{m-1} + (-2\log x)^{m-1}\right) \int_{u_{j-1}(x)}^{u_j(x)} e^{-f(v)/2-v}v^{-1}\, dv \\
& \leq  -m2^{m-1}(\log x) e^{x/2}\int_{0}^{\infty} e^{-f(v)/2-v}v^{m-2}\, dv  + m4^{m-1}(-\log x)^m e^{x/2}\int_{0}^{\infty} e^{-f(v)/2-v}v^{-1}\, dv.
\end{align*}
We claim that both the integrals above are finite since $m \geq 1$. Indeed, the integrands are clearly integrable over $(1,\infty)$ since $f \geq 0$. On the otherhand, for $v \in (0,1)$, we have $f(v)-f(1) = \int_{v}^1 z^{-1}e^{-z} \, dz \geq e^{-1} \int_{v}^1 z^{-1}\, dz = -(\log v)/e$ and hence $e^{-f(v)/2-v}v^{\alpha} \leq e^{-f(1)/2}v^{1/(2e)+m-2} \leq e^{-f(1)/2}v^{1/(2e)-1}$, which is integrable over $(0,1)$. Taking $x \uparrow 1$, we conclude that $S_U(x)-S_L(x) \to 0$ as $x \uparrow 1$. Moreover $0 \leq u_0(x) \leq -\log x$ and hence $u_0(x) \to 0$ as $x \uparrow 1$, yielding that
$$S_M(x) \longrightarrow \dfrac{\sqrt{e}}{2} \int_{0}^{\infty} e^{-f(v)/2-v}v^{m-1}\, dv  , \; \text{ as } x \uparrow 1.$$ 
Combining all of these computations we can conclude that 
$$ (-\log x)^m(1-x)^{-1/2} \mathbb{E}  L(\Pi_x)^m2^{-c_{-1}(\Pi_x)} \longrightarrow \gamma_m := \dfrac{\sqrt{e}}{2} \int_{0}^{\infty} e^{-f(v)/2-v}v^{m-1}\, dv.$$
The integral above is finite according to our previous discussion. Setting $\upsilon_{n,m} := \mathbb{E}  L(\Pi_n)^m2^{-c_{-1}(\Pi_n)}$, for all $n \geq 0$ and recalling that $-\log x \sim (1-x)$ as $x \to 1$, we conclude the following.
$$ \sum_{n \geq 0} (1-x)x^n\upsilon_{n,m} \sim \dfrac{\gamma_m}{(1-x)^{m-1/2}}, \text{i.e.}, \sum_{n \geq 0} x^n(\upsilon_{n,m}-\upsilon_{n-1,m}) \sim \dfrac{\gamma_m}{(1-x)^{m-1/2}}, \;\text{ as } x \uparrow 1.$$
This is ripe for using Hardy-Littlewood Tauberian Theorem in the form given in \cite[Theorem 7.4]{korevaar}; provided we have 
$$ \inf_{n \geq 1} n^{3/2-m}(\upsilon_{n,m}-\upsilon_{n-1,m}) > - \infty,$$
we can conclude that 
$$ \upsilon_{n,m} \sim \dfrac{\gamma_m}{\Gamma(m+1/2)}n^{m-1/2}, \; \text{ as } n \to \infty,$$
completing the proof. To prove the uniform bound, we again use the function $\omega_n$ defined earlier in this section. Note that 
\begin{align*}
\upsilon_{n,m}-\upsilon_{n-1,m} &= \mathbb{E}  L(\Pi_n)^m2^{-c_{-1}(\Pi_n)} - \mathbb{E}  L(\omega_n(\Pi_n))^m2^{-c_{-1}(\omega_n(\Pi_n))} \\
&= \mathbb{E}  L(\Pi_n)^m2^{-c_{-1}(\omega_n(\Pi_n))-\mathbbm{1}(R_n=2)} - \mathbb{E}  L(\omega_n(\Pi_n))^m2^{-c_{-1}(\omega_n(\Pi_n))} \\
& \geq \mathbb{E}  L(\omega_n(\Pi_n))^m2^{-c_{-1}(\omega_n(\Pi_n))}\left(2^{-\mathbbm{1}(R_n=2)} -1 \right) \\
& = - \dfrac{1}{2n} \mathbb{E}L(\omega_n(\Pi_n))^m2^{-c_{-1}(\omega_n(\Pi_n))}c_1(\omega_n(\Pi_n)) \geq - \dfrac{n^{m-1}}{2} \mathbb{E}2^{-c_{-1}(\omega_n(\Pi_n))}c_1(\omega_n(\Pi_n)). 
\end{align*}
Thus it is enough to show that $\sup_{n \geq 1} \sqrt{n}\mathbb{E}2^{-c_{-1}(\Pi_n)}c_1(\Pi_n) < \infty.$ Plug-in $t=1$ in \Cref{diff}  to obtain the following.
$$ \sum_{n \geq 0} x^n \mathbb{E}2^{-c_{-1}(\Pi_n)}c_1(\Pi_n) = e^{-x/2}x(1-x)^{-1/2} = \left( \sum_{k \geq 0} \dfrac{(-1)^k}{2^k k!}x^{k+1}\right)\left( \sum_{k \geq 0} 2^{-2k}{2k \choose k} x^k \right).$$ 
Comparing co-efficients we obtain,
\begin{align*}
\mathbb{E}2^{-c_{-1}(\Pi_n)}c_1(\Pi_n)  = \sum_{k =1}^n \dfrac{(-1)^{k-1}}{2^{k-1}(k-1)!}2^{-2(n-k)}{2(n-k) \choose (n-k)} \leq 2^{-2(n-1)}\left( 2(n-1)\choose (n-1) \right) \sim \dfrac{1}{\sqrt{\pi n}},
\end{align*}
where the inequality follows from the observation that the absolute values of the summands above are non-increasing in $k$ for a fixed large $n$. This completes the proof.
\end{proof}

\begin{remark}
One can try to extend the idea of this section for any $r$-regular bipartite graphs with $r \geq 3$. In order to do that, we fix $r$ colours $\left\{\texttt{col}_1, \ldots,\texttt{col}_r\right\}$ and let $\mathcal{M}^*(n,r)$ be the set of all edge-coloured (with the given $r$ colours) $r$-regular bipartite graphs.  Since every $r$-regular bipartite graph allows an edge-colouring with $r$ colours, see \cite{konig}, every element of $\mathcal{M}(n,r)$ is represented in $\mathcal{M}^*(n,r)$. It is also easy to see that $\operatorname{card}(\mathcal{M}^*(n,r))=(n!)^r$ and there is a bijection between $\mathcal{M}^*(n,r)$ and $\mathscr{S}_n^{\times r}$. This allows us to easily simulate from $\mathcal{M}^*(n,r)$ using random permutations. The difficulty arises from counting the number of edge-colouring by $r$ colours possible for a particular matrix/graph $A$ in $\mathcal{M}(n,r)$. In case of $r=2$, this number was determined by the component spectrum of $A$, which probably will not be the case for $r \geq 3$ since the number of edge colourings of a connected graph $A$ in $\mathcal{M}(n,r)$ probably is not determined by $n$ and $r$ only, but further structures of $A$. This makes the importance sampling formula in this case much more complicated to derive.
\end{remark}

\section{Acknowledgment}
The author would like to express his sincere gratitude to Prof.~Persi Diaconis for introducing the author to the wonderful book of~\cite{stanenu} and for his insightful suggestions throughout the working stage of this paper.

\end{document}